\documentclass[11pt]{amsart}
\usepackage[margin=1in]{geometry}

\usepackage{amssymb}
\usepackage{amsthm}
\usepackage{amsmath}
\usepackage{mathrsfs}
\usepackage{amsbsy}
\usepackage[all]{xy}
\usepackage{bm}
\usepackage{hyperref}
\usepackage{tikz}
\usepackage{array}
\usepackage{float}
\usepackage{enumerate}
\usepackage{xcolor}
\usepackage{hhline}
\usepackage{mathtools}
\allowdisplaybreaks
\usepackage[noadjust]{cite}

\usepackage{caption}
\usepackage{subcaption}
\usepackage{tabu}
\usepackage{diagbox}
\usepackage{tikz}
\usepackage{bbm}
\usepackage{booktabs}

\usepackage[noabbrev,capitalise]{cleveref}

\newenvironment{enumerate*}%
  {\begin{enumerate}[(I)]%
    \setlength{\itemsep}{10pt}%
    \setlength{\parskip}{0pt}}%
  {\end{enumerate}}

\newtheorem{theorem}{Theorem}[section]
\newtheorem{proposition}[theorem]{Proposition}

\newtheorem{conjecture}[theorem]{Conjecture}

\newtheorem{lemma}[theorem]{Lemma}
\newtheorem{claim}[theorem]{Claim}

\theoremstyle{definition}
\newtheorem{definition}[theorem]{Definition}
\newtheorem{remark}[theorem]{Remark}

\newcommand{\bp}{\mathbf{p}}
\newcommand{\bd}{\mathbf{d}}
\newcommand{\bn}{\mathbf{n}}
\newcommand{\bt}{\mathbf{t}}
\newcommand{\bu}{\mathbf{u}}

\DeclareMathOperator{\Span}{span}
\DeclareMathOperator{\IS}{IS}

\title{Graham's rearrangement conjecture beyond the rectification barrier}

\author[]{Benjamin Bedert}
\address[]{Mathematical Institute, Andrew Wiles Building, University of Oxford, Radcliffe
Observatory Quarter, Woodstock Road, Oxford, OX2 6GG, UK.}
\email{benjamin.bedert@maths.ox.ac.uk}

\author[]{Noah Kravitz}
\address[]{Department of Mathematics, Princeton University, Princeton, NJ 08540, USA}
\email{nkravitz@princeton.edu}

\begin{document}

\begin{abstract}
A 1971 conjecture of Graham (later repeated by Erd\H{o}s and Graham) asserts that every set $A \subseteq \mathbb{F}_p \setminus \{0\}$ has an ordering whose partial sums are all distinct.  We prove this conjecture for sets of size $|A| \leqslant e^{(\log p)^{1/4}}$; our result improves the previous bound of $\log p/\log \log p$.  One ingredient in our argument is a structure theorem involving dissociated sets, which may be of independent interest.
\end{abstract}

\maketitle

\section{Introduction}
\subsection{Main result}
Let $A$ be a finite subset of an abelian group.  We say that an ordering $a_1, \ldots, a_{|A|}$ of $A$ is \emph{valid} if the partial sums $a_1, a_1+a_2, \ldots, a_1+a_2+\cdots+a_{|A|}$ are all distinct.  In 1971, Graham conjectured that every set of non-zero elements of $\mathbb{F}_p$ has a valid ordering.

\begin{conjecture}[\cite{graham}]\label{conj:main}
Let $p$ be a prime.  Then every subset $A \subseteq \mathbb{F}_p \setminus\{0\}$ has a valid ordering.
\end{conjecture}

This conjecture also appeared in a 1980 book of Erd\H{o}s and Graham \cite{EG}, and a very similar conjecture for finite cyclic groups is due to Alspach (see \cite{BH}).

The main avenue of attack on Graham's conjecture has been to show that its conclusion holds when $A$ is small.  Until recently, the published world record had established Graham's conjecture for sets $A$ of size at most $12$ (see, e.g., the discussion in \cite{kravitz,CP}).  Earlier this year, the second author \cite{kravitz} used a simple rectification argument to show that Graham's conjecture holds for all sets $A$ of size $|A| \leqslant \log p/\log\log p$; Will Sawin \cite{will} had independently proven a comparable bound, using roughly similar ideas, in a 2015 MathOverflow post.  The purpose of the present paper is to prove Graham's conjecture for sets $A$ of up to quasi-polynomial size.

\begin{theorem}\label{thm:main}
The following holds for every constant $c>0$.  Let $p$ be a large prime.  Then every subset $A \subseteq \mathbb{F}_p \setminus\{0\}$ of size
$$ |A| \leqslant e^{c(\log p)^{1/4}}$$
has a valid ordering.
\end{theorem}
We have not made a serious effort to optimize the exponent $1/4$, but the quasi-polynomial shape of this bound does appear as a natural barrier in several parts of our argument.  We also mention that the conclusion of Theorem \ref{thm:main} still holds, with a nearly identical proof, if $\mathbb{F}_p$ is replaced by any abelian group with no non-zero elements of order strictly smaller than $p$.

The proof strategy for Theorem \ref{thm:main} is motivated by the argument in \cite{kravitz}.  (The argument in \cite{will} seems less well-suited to generalization.)  Two new ingredients are the theory of dissociated sets (from additive combinatorics) and probabilistic tools.  One of our intermediate results (see Theorem \ref{thm:structure} below) is a structure theorem involving dissociated sets, which may be of independent interest.

\subsection{Proof sketch and organization}
We say that an ordering of $A$ is \emph{two-sided valid} if no proper nonempty subinterval sums to zero; this condition is slightly stronger than $A$ being valid.  As in \cite{kravitz}, we will prove Theorem \ref{thm:main} with two-sided valid orderings.

Let us briefly recall the main ideas of \cite{kravitz}.  Let $A \subseteq \mathbb{F}_p \setminus \{0\}$ be a subset of size $|A| \leqslant \log p/2\log\log p$.  Using the pigeonhole principle, one can find some $\lambda \in \mathbb{F}_p^\times$ such that the dilate $\lambda \cdot A$ is contained in the interval $(-p/|A|,p/|A|)$.  Since sums of elements of $\lambda \cdot A$ have no ``wrap-around'', we can interpret $\lambda \cdot A$ as a subset of $\mathbb{Z} \setminus \{0\}$; this process is known as ``rectification''.  Finally, in the integer setting, one can use induction on $|A|$ to find a two-sided valid ordering in which all of the positive elements appear before all of the negative elements.

Our proof of Theorem \ref{thm:main} proceeds in four main steps.  The first step is showing that every subset of $\mathbb{F}_p$ can be decomposed into a union of large dissociated sets and a rectifiable residual set.  (A dissociated set is a set all of whose subset sums are distinct; see below.)  The residual set can be broken into ``positive'' and ``negative'' sets.  We will aim to find a two-sided valid ordering consisting of the positive elements, then the elements of the dissociated sets, then the negative elements.

The second step is ordering the positive and negative elements.  Following \cite{kravitz}, we inductively construct these orderings in order to avoid zero-sum intervals that begin in the positive region and end in the negative region.  We take advantage of some flexibility in the argument from \cite{kravitz} in order to prepare for ``potential'' zero-sum intervals with one endpoint in the positive region or negative region and the other endpoint very close to one of the edges of the dissociated region.

The third and fourth steps concern ordering the elements of the dissociated sets.  The main idea is that in a uniformly random ordering of a dissociated set of size $R$, the sum of the first $k$ elements is uniformly distributed on $\binom{R}{k}$ different values.  Since the probability of assuming any particular value is very small, the probability of this initial segment forming the end of a zero-sum interval is also very small.  This na\"ive random strategy essentially works for handling sets $A$ of size up to $(\log p)^{3/2}$ (which breaks the ``rectification barrier'' of \cite{kravitz}), but we must employ a more elaborate random procedure in order to reach the threshold in Theorem \ref{thm:main}.  In particular, it becomes important to distinguish between the ``borders'' and ``interiors'' of the orderings of the dissociated sets.  The third step of the proof is randomly splitting and then reordering the dissociated sets, and the fourth step is choosing a (suitably) random ordering for the elements within each dissociated set.

We carry out these four steps in Sections \ref{sec:structure}, \ref{sec:PN}, \ref{sec:splitting}, and \ref{sec:random}, respectively, and then we make some concluding remarks and pose several open problems in Section \ref{sec:remarks}.

\section{Notation and parameters}
Before jumping into the proofs, we set a few pieces of notation.
\begin{itemize}
    \item We use $\mathbb{F}_p$ to denote the field with $p$ elements, and for us $p$ will always be a large prime.
    \item We denote dilation by $\lambda \cdot A\vcentcolon=\{\lambda a: a \in A\}$.
    \item We denote the restricted sumset by $B\hat{+}B\vcentcolon=\{b+b':b,b'\in B \text{ and } b\neq b'\}$.
    \item Let $\sum_{= M}(S)\vcentcolon=\{\sum_{s\in S'}s:S'\subseteq S,|S'|= M\}$ denote the set of all sums of exactly $M$ elements of $S$.  Likewise, let $\sum_{ \leqslant M}(S)\vcentcolon=\{\sum_{s\in S'}s:S'\subseteq S,|S'|\leqslant M\}$ denote the set of all sums of at most $M$ elements of $S$, and let $\sum_{ \geqslant M}(S)\vcentcolon=\{\sum_{s\in S'}s:S'\subseteq S,|S'|\geqslant M\}$ denote the set of all sums of at least $M$ elements of $S$.
    \item For a sequence $\mathbf{b}=b_1, \ldots, b_r$, let $\IS(\mathbf{b})\vcentcolon=\{b_1+\cdots+b_j: 0 \leqslant j \leqslant r\}$ denote the set of initial segment sums of $\mathbf{b}$, and let $\overline{\mathbf{b}}\vcentcolon=b_r, \ldots, b_1$ denote the reverse of $\mathbf{b}$.
\end{itemize}

We use standard asymptotic notation.  We write $f=O(g)$ or $f \ll g$ if there is a universal constant $C>0$ such that $|f| \leqslant Cg$.  If $f$ is non-negative and $f=O(g)$, then we also write $g=\Omega(f)$.  We write $f \asymp g$ when $f \ll g$ and $g \ll f$. 
 Finally, we write $f(p)=o(g(p))$ if $\lim_{p \to \infty} f(p)/g(p)=0$.

When there is no risk of confusion, we sometimes omit floor functions in calculations for typographical clarity.

Let us also record a few parameters that we will carry through our proofs.
\begin{itemize}
    \item Our set $A$ will have size $|A| \leqslant e^{c(\log p)^{1/4}}$ for some absolute constant $c>0$.
    \item We define the \emph{rectification threshold} for a subset $A\subseteq\mathbb{F}_p$ to be
\begin{align*}
    R=R(A)\vcentcolon=c_1\max\left((\log p)^{1/2},\frac{\log p}{\log |A|}\right),
\end{align*}
where $c_1>0$ is a sufficiently small absolute constant.
    \item The \emph{border width} is $K\vcentcolon=c_2 R^{1/3}$, for yet another absolute constant $c_2>0$.
    \item We will use $s$ (and later $u$) to denote the number of dissociated sets in our decomposition of the set $A$.  The precise values of $s,u$, which are of no importance (besides the trivial bound $s,u \leqslant |A|$), will vary over the course of the proofs.
    \item We will always use $\delta_j$ to denote the sum of the elements of the set $D_j$, and we will always use $\tau_j$ to denote the sum of the elements of the set $T_j$; when applicable, we will also write $\delta:=\sum_j \delta_j$.
\end{itemize}

Finally, we reiterate that a sequence $b_1, \ldots, b_t$ is \emph{two-sided valid} if
$$b_i+\cdots+b_j \neq 0 \quad \text{for all $1 \leqslant i<j \leqslant t$ with $(i,j) \neq (1,t)$}.$$

\section{Structure theorem}\label{sec:structure}
Let $G$ be an abelian group.  A subset $D=\{d_1, \ldots, d_r\} \subseteq G$ is \emph{dissociated} if
$$\epsilon_1 d_1+\cdots +\epsilon_r d_r \neq 0 \quad \text{for all $(\epsilon_1, \ldots, \epsilon_r) \in \{-1,0,1\}^r \setminus \{(0,\ldots, 0)\}$}.$$
Equivalently, $D$ is dissociated if all of the $2^{|D|}$ subset sums of $D$ are distinct. The \emph{dimension} of a subset $B \subseteq G$, written $\dim(B)$, is the size of the largest dissociated set contained in $B$.  One should think of sets of small dimension as being highly ``constrained''.

\begin{lemma}\label{lem:dissociated-set-generator}
Let $B \subseteq G$ be a finite subset of an abelian group.  If $D$ is a maximal dissociated subset of $B$, then
$$B \subseteq \Span(D) \vcentcolon= \left\{\sum_{d\in D}\varepsilon_d d: \varepsilon_d\in\{-1,0,1\}\right\}.$$
\end{lemma}

\begin{proof}
The maximality of $D$ ensures that for every element $b\in B\setminus D$, the set $\{b\}\cup D$ is not dissociated; rearranging then gives the desired expression for $b$ as an element of $\Span(D)$.
\end{proof}

The following lemma says that sets of sufficiently small dimension can always be ``rectified''. To make this precise, we define for each (nonempty) subset $A\subseteq\mathbb{F}_p$ the parameter
\begin{align}\label{Rdefinition}
    R=R(A)\vcentcolon=c_1\max\left((\log p)^{1/2},\frac{\log p}{\log |A|}\right),
\end{align}
where $c_1$ is a sufficiently small absolute constant. 

\begin{lemma}\label{lem:small-dim-recitification}
If $B \subseteq \mathbb{F}_p$ is a nonempty subset of dimension $\dim(B)< R=R(B)$, then there is some $\lambda \in \mathbb{F}_p^\times$ such that the dilate $\lambda\cdot B$ is contained in the interval $(-\frac{p}{100|B|},\frac{p}{100|B|})$.
\end{lemma}

\begin{proof}
Let $D$ be a maximal dissociated subset of $B$, so that $|D|= \dim(B)$. Lemma~\ref{lem:dissociated-set-generator} tells us that $B\subseteq \operatorname{span}(D)$.  Consider the set
$$\{(\lambda d/p)_{d\in D}: \lambda \in \mathbb{F}_p\} \subseteq (\mathbb{R}/\mathbb{Z})^{\dim(B)}.$$
The pigeonhole principle provides some distinct $\lambda_1, \lambda_2 \in \mathbb{F}_p$ such that $\|\lambda_1 d/p-\lambda_2 d/p\|_{\mathbb{R}/\mathbb{Z}} \leqslant p^{-1/\dim(B)}$ for all $d\in D$.  Set $\lambda\vcentcolon=\lambda_1-\lambda_2 \in \mathbb{F}_p^\times$, so that $\lambda d \in [-p^{1-1/\dim(B)},p^{1-1/\dim(B)}]$ for all $d\in D$.  Since $B\subseteq \operatorname{span}(D)$, we have
$$\lambda \cdot B \subseteq [-\dim(B)p^{1-1/\dim(B)},\dim(B)p^{1-1/\dim(B)}].$$
It remains only to show that $\dim(B)p^{1-1/\dim(B)}<p/(100|B|)$, i.e., that $100|B|\dim(B)<p^{1/\dim(B)}$, as long as $c_1$ is chosen to be sufficiently small.  When $\log |B|<(\log p)^{1/2}$, this inequality follows from $\dim(B)\leqslant R=c_1\log p /\log |B|$. When $\log |B|\geqslant (\log p)^{1/2}$, the desired inequality follows from $\dim(B) \leqslant R=c_1(\log p)^{1/2}$ and $|B| \leqslant 3^{\dim(B)}$.
\end{proof}
\begin{remark}
For applications in this paper, we will always work with sets of size at most $e^{c(\log p)^{1/4}}$, in which case the previous lemma says that every set $B$ of size at most $e^{c(\log p)^{1/4}}$ with $\dim(B)<c_1(\log p)^{3/4}$ is rectifiable.  We opted to prove Lemma \ref{lem:small-dim-recitification} for arbitrary sets $B\subseteq\mathbb{F}_p$, however, so that we could state the structural results in the rest of this section in full generality.  These results are nontrivial for sets $B$ of all sizes since the rectification threshold always satisfies $R(B)\gg(\log p)^{1/2}$. 
\end{remark}

We can combine these two lemmas to obtain a decomposition of \emph{any} subset of $\mathbb{F}_p$ into large dissociated sets and a residual set that (after suitable dilation) is contained in a small interval around $0$.  We shall from now on simply write $R$ for $R(A)$.  The following theorem bears many similarities to an argument of Bourgain \cite{bourgain} from a different context.

\begin{theorem}\label{thm:structure}
Every subset $A \subseteq \mathbb{F}_p$ can be partitioned as
\begin{equation}
    A=D_1 \cup \cdots \cup D_s \cup E,
    \label{structure1}
\end{equation}
where the following holds:
\begin{enumerate}[(i)]
    \item each $D_j$ is a dissociated set of size $|D_j|\asymp R $;
    \item $|E|\geqslant R/2$ if $s>0$;
    \item there is some $\lambda \in \mathbb{F}_p^\times$ such $\lambda \cdot (E \cup \{\delta\}) \subseteq (-\frac{p}{90(|E|+1)},\frac{p}{90(|E|+1)})$, where $\delta\vcentcolon=\sum_{j=1}^s \sum_{d \in D_j} d$ is the sum of all of the elements in the dissociated sets.
\end{enumerate}
\label{structuretheorem1}
\end{theorem}

\begin{proof}
Start with $E=A$. As long as $\dim(E) \geqslant R$, iteratively remove a dissociated subset of size $R/2$, so that at each step the set $E$ of remaining elements has size $|E|\geqslant R/2$.  Once we reach a residual set $E$ of dimension smaller than $R$, Lemma~\ref{lem:small-dim-recitification} (applied to $E \cup \{\delta\}$, where $\delta$ is the sum of all of the dissociated elements removed) provides the desired $\lambda \in \mathbb{F}_p^\times$.
\end{proof}

We will, of course, apply this theorem to the set $A$ for which we are trying to find a two-sided valid ordering. 
If the number $s$ of dissociated sets happens to be $0$, then the entire set $A$ is rectifiable and therefore has a two-sided valid ordering by \cite{kravitz} (see the discussion in the proof sketch).  Thus, we will restrict our attention to the case where $s\geqslant 1$ (so in particular $|E| \gg R$ from (ii)).  The presence of a large dissociated set allows us to obtain a more detailed structural result.

\begin{proposition}\label{prop:refined-structure}
For every nonempty subset $A \subseteq \mathbb{F}_p \setminus \{0\}$, there is some $\lambda \in \mathbb{F}_p^\times$ such that $\lambda \cdot A$ can be partitioned as
     $$\lambda\cdot A = P\cup N\cup(\cup_{j=1}^{s} D_j),$$
    where
\begin{enumerate}[(i)]
    \item the ``positive'' set $P$ is contained in $(0,\frac{p}{4|P\cup N|})$, the ``negative set'' $N$ is contained in $(-\frac{p}{4|P\cup N|},0)$, and the element $\delta\vcentcolon=\sum_{j=1}^{s} \sum_{d \in D_j} d$ is contained in $(-\frac{p}{4},\frac{p}{4})$;
\end{enumerate}
and the following also holds if $s>0$:
\begin{enumerate}[(i)]
    \item[(ii)] $P \cup N$ is nonempty, and each $D_j$ is a dissociated set of size $|D_j|\asymp R$, where the implied constant is absolute;
    \item[(iii)] $\delta \notin \{0\}\cup -P\cup -N$, and moreover $\delta \neq -\sum_{p \in P}p$ if $N$ is nonempty and $\delta \neq -\sum_{n \in N}n$ if $P$ is nonempty;
    \item[(iv)] $D_1 \cup D_s \cup \{\delta\}$ is a dissociated set;
    \item [(v)] $|D_1|=|D_s|$.
\end{enumerate}
 \end{proposition}
Before proving this proposition, we make a simple but powerful observation about absorbing elements into dissociated sets.

\begin{lemma}\label{absorptionlemma}
Let $G$ be an abelian group, and let $D_1 \cup D_2$ be a partition of a dissociated subset of $G$.  For every element $x \in G \setminus \{0\}$, either $D_1 \cup \{x\}$ or $D_2 \cup \{x\}$ is dissociated.
\end{lemma}

\begin{proof}
Assume for the sake of contradiction that neither $D_1 \cup \{x\}$ nor $D_2 \cup \{x\}$ is dissociated.  Since $D_1$ is dissociated, the failure of $D_1 \cup \{x\}$ to be dissociated implies that $x\in \Span(D_1)$; similarly, $x\in \Span(D_2)$. So $\Span(D_1)\cap\Span(D_2)$ contains a non-zero element, contradicting the assumption that $D_1 \cup D_2$ is dissociated.
\end{proof}

Iterating this observation, we find that if $B$ is a set of size $t$ and $D_1 \cup \cdots \cup D_{t+1}$ is a partition of a dissociated set, then it is always possible to add the elements of $B$ to the dissociated sets $D_j$ in such a way that the sets remain dissociated.

We will also make use of the trivial lower bound for the size of a restricted sumset in $\mathbb{Z}$: If $B \subseteq \mathbb{Z}$ is a finite set, then $|B\hat{+}B|\geqslant 2|B|-3$.  We are now ready to prove Proposition \ref{prop:refined-structure}.  The choice of numerical constants appearing in the proof is not important.
 
 \begin{proof}[Proof of Proposition \ref{prop:refined-structure}.]
     To start, Theorem \ref{structuretheorem1} provides some $\lambda \in \mathbb{F}_p^\times$ and a decomposition $$\lambda \cdot A=D_1 \cup \cdots \cup D_{s} \cup E,$$ where each $D_j$ is a dissociated set of size $\asymp R$ and we have $E \cup \{\delta\} \subseteq (-\frac{p}{90(|E|+1)},\frac{p}{90(|E|+1)})$, for $\delta\vcentcolon=\sum_{j=1}^{s} \sum_{d \in D_j} d$.  Set $P\vcentcolon=E\cap (0,p/4|E|)$ and $N\vcentcolon=E\cap(-p/4|E|,0)$.  If $s=0$, then we have already obtained the desired decomposition of $\lambda \cdot A$, so for the remainder of the proof we assume that $s \geqslant 1$.  By replacing $\lambda$ with $-\lambda$ if necessary, we may assume that $|P|\geqslant |N|$.  In particular, since $|E|\gg R$, this implies that $|P|\gg R$.    

     We remark that once we have a decomposition satisfying conditions (i)--(iii), we can modify the decomposition to satisfy (iv) and (v) as follows.  Split $D_1$ into $2$ parts $D_1^{(1)},D_1^{(2)}$ each of size $\asymp R$.  Lemma \ref{absorptionlemma} ensures that either $D_1^{(1)} \cup \{\delta\}$ or $D_1^{(2)} \cup \{\delta\}$ is dissociated; without loss of generality, assume that $D_1^{(1)} \cup \{\delta\}$ is dissociated.  Then further split $D_1^{(1)}$ into $2$ parts $D_1^{(3)}, D_1^{(4)}$ each of size $\lfloor |D_1^{(1)}|/2 \rfloor \asymp R$, add the leftover element of $D_1^{(1)}$ to $D_1^{(2)}$ if $|D_1^{(1)}|$ was odd, and replace the sequence of sets $D_1, \ldots, D_s$ by the sequence $D_1^{(3)}, D_1^{(2)}, D_2, D_3, \ldots, D_s, D_1^{(4)}$.  This new sequence satisfies (iv) and (v).  The remainder of the proof is devoted to finding a decomposition satisfying conditions (i)--(iii).
     
     We will later apply sumset inequalities involving $P,N$, and we will need $P,N$ to be not-too-small so that we have ``room'' for sumsets to expand.  In anticipation of this, we begin by reducing to the case where $N$ is either empty or of size at least $10$.  Suppose that $0<|N|<10$.  Note that $\sum_{n \in N}n \in (-p/90,p/90)$.  Split $D_1$ into $|N|+1 \leqslant 10$ sets each of size $\asymp R$; the remark before the proof ensures that we can absorb all of the elements of $N$ into these dissociated sets, and this procedure changes the value of $\delta$ by at most $p/90$.  Notice that each newly formed $D_j$ still has size $\asymp R$, and that we still have $P \subseteq (0,\frac{p}{40|P \cup N|})$, $N \subseteq (-\frac{p}{40|P \cup N|},0)$, and $\delta \in (-p/40,p/40)$.
     
We now consider two cases depending on the size of $N$. First, suppose that $N=\emptyset$, and recall that $|P|\gg R$.
Since $P$ is rectifiable (i.e., Freiman-isomorphic to a subset of $\mathbb{Z}$), the trivial lower bound for restricted sumsets in integers gives $$|P\hat{+}P|\geqslant 2|P|-3>|P|+2,$$ and hence we can find distinct $p_1,p_2\in P$ such that $p_1+p_2+\delta\notin \{0,-\sum_{n\in N}n\}\cup -P$.  Splitting $D_1$ and absorbing $p_1,p_2$ with the help of Lemma \ref{absorptionlemma} as above yields a new decomposition of $A$ where the sum of all of the elements in the dissociated sets is $p_1+p_2+\delta$ and hence conditions (ii) and (iii) are satisfied. This procedure also changes the value of $\delta$ by at most $p/90$ (say), so we obtain the desired decomposition of $\lambda \cdot A$.
        
Finally, suppose that $|N|\geqslant 10$, and recall that we also have $|P| \geqslant |N| \geqslant 10$.  By absorbing $5$ arbitrary elements of $N$ into $D_1$, we may assume that $|P| \geqslant |N|-5$.  Since we still have $|N|>1$, there is some $n_1 \in N$ such that $\delta+n_1 \neq -\sum_{p \in P}p$; as above, we absorb $n_1$ into $D_1$, so that the final part of condition (iii) is satisfied.  Now we have
$$|P \hat + P| \geqslant 2|P|-3>|P|+|N|+2,$$
so there are distinct $p_1,p_2 \in P$ such that $$\delta+p_1+p_2 \notin \{0,-\sum_{n \in N}n\} \cup -P \cup -N.$$
Absorbing $p_1,p_2$ into $D_1$ gives the desired decomposition of $\lambda \cdot A$.  (For the final part of condition (iii), note that this last step preserves the property $\delta+\sum_{p \in P}p \neq 0$.) 
\end{proof}

\section{Ordering $P$ and $N$}\label{sec:PN}
With Proposition \ref{prop:refined-structure} in hand, we can say a bit more about the remainder of the proof of Theorem \ref{thm:main}.  We will aim to find orderings $\bp$ of $P$, $\bn$ of $N$, and $\bd$ of $\cup_j D_j$ such that $\overline{\bp}, \bd, \bn$ is a two-sided valid ordering of $A$.  Of course, we will need each of the three orderings to be two-sided valid on its own, and we will need to avoid creating zero-sum intervals when we concatenate them.

Condition (i) from Proposition \ref{prop:refined-structure} means that the problem of constructing $\bp$ and $\bn$ naturally lives in the integers rather than in $\mathbb{F}_p$, as follows.  Identify $\delta$ and the elements of $P \cup N$ with elements of $(-p/4,p/4) \subseteq \mathbb{Z}$ in the natural way, and note that sums of these elements can be computed equivalently in $\mathbb{F}_p$ and in $(-p/2,p/2) \subseteq \mathbb{Z}$ because the sums in $\mathbb{F}_p$ do not exhibit any wrap-around.  Likewise, for any ordering $\bd$ of $\cup_j D_j$, we can identify $\IS(\bd)$ and $\IS(\overline{\bd})$ with subsets of $(-p/2,p/2) \subseteq \mathbb{Z}$.  Now we observe that the ordering $\overline \bp,\bd,\bn$ is two-sided valid if and only if the ordering $\overline\bp,\delta, \bn$ is two-sided valid in the integers, the ordering $\bd$ is two-sided valid in $\mathbb{F}_p$, and $\IS(\bp) \cap -\IS(\bd)=\IS(\bn) \cap -\IS(\overline{\bd})=\emptyset$; the key point is that the first condition lives entirely in the integers, the second condition does not concern $\bp$ and $\bn$, and the third condition lives in the integers for each fixed choice of $\bd$.

We will later choose $\bd$ randomly, but it turns out that we can model $-\IS(\bd),-\IS(\overline{\bd})$ by somewhat larger deterministic sets that encode all of the ``potentially important'' intersections with $\IS(\bp),\IS(\bn)$ (respectively); it will suffice to ensure that $\IS(\bp), \IS(\bn)$ have fairly small intersections with these deterministic sets.  With this in mind, the main result of this section is as follows.



\begin{proposition}\label{prop:ordering-P-N}
Let $P \subseteq (0,\infty)$ and $N \subseteq (-\infty,0)$ be finite sets of integers, and let $\delta>0$ be a positive integer not contained in $-N$; moreover, assume that $ \delta \neq -\sum_{n \in N}n$ if $P \neq \emptyset$.  Let $Y^+_1, \ldots, Y^+_m,Y^-_1, \ldots, Y^-_m \subseteq \mathbb{Z}$ be finite sets.  Then there are orderings $\bp$ of $P$ and $\bn$ of $N$ such that $\overline \bp, \delta, \bn$ is two-sided valid and we have
\begin{align}\label{ISp-bound}
    |\IS(\bp) \cap Y^+_j| \leqslant \inf_{L \in \mathbb{N}} \left(\frac{|Y^+_j|}{L}+L+4+4\sum_{i=1}^{j-1}|Y^+_i|\right)
\end{align}
and 
\begin{align}\label{ISn-bound}
    |\IS(\bn) \cap Y^-_j| \leqslant \inf_{L \in \mathbb{N}} \left(\frac{|Y^-_j|}{L}+L+4+4\sum_{i=1}^{j-1}|Y^-_i|\right)
\end{align}
for all $1 \leqslant j \leqslant m$.
\end{proposition}

In \cite{kravitz}, the first author presented a simple algorithm for inductively constructing a two-sided valid ordering of any finite subset of $\mathbb{Z} \setminus \{0\}$.  Let us quickly review this algorithm since it forms the basis for the proof of Proposition \ref{prop:ordering-P-N}.  Suppose $P \subseteq (0,\infty)$ and $N \subseteq (-\infty,0)$ are finite sets of integers; we want to produce orderings $\bp$ of $P$ and $\bn$ of $N$ such that $\overline{\bp},\bn$ is two-sided valid.  We will construct the sequences $\bp=p_1, \ldots, p_{|P|}$ and $\bn=n_1, \ldots, n_{|N|}$ from the larger indices to the smaller indices.  For the first step, consider the sign of $\sum_{n \in N}n+\sum_{p \in P}p$.  Suppose that this sum is non-negative; we will choose the value of $p_{|P|}$ as follows.  There is some $p^* \in P$ such that $\sum_{n \in N}n+\sum_{p \in P \setminus \{p^*\}} p \neq 0$, and we choose this $p^*$ to be our $p_{|P|}$.  This choice ensures that if $\overline{\bp'},\bn$ is a two-sided valid ordering of the remaining elements $P \setminus \{p_{|P|}\}, N$, then $p_{|P|},\overline{\bp'},\bn$ is the desired two-sided valid ordering of $P,N$: Any interval containing $p_{|P|}$ and not containing any of $\bn$ clearly has strictly positive sum; any proper interval containing both $p_{|P|}$ and some elements of $\bn$ must contain all of $P$ and hence has strictly positive sum by our assumption that $\sum_{n \in N}n+\sum_{p \in P}p \geqslant 0$; the intervals strictly contained in $\overline{\bp'},\bn$ are all non-zero-sum by assumption; and the interval consisting of all of $\overline{\bp'},\bn$ has non-zero sum by our choice of $p^*$.  If instead $\sum_{n \in N}n+\sum_{p \in P}p<0$, then we choose $n_{|N|}$ analogously.  With this first step complete, we throw out the already-chosen element $p_{|P|}$ or $n_{|N|}$ and repeat this process with the remaining elements.  This procedure produces the desired orderings $\bp,\bn$.

The above algorithm has a lot of slack, in the sense that at each step there are many possible choices.  To exploit this slack and prove Proposition \ref{prop:ordering-P-N}, we will employ a modified algorithm that greedily avoids partial sums of $\bp$ lying in the $Y^+_j$'s and partial sums of $\bn$ lying in the $Y^-_j$'s.

\begin{proof}[Proof of Proposition \ref{prop:ordering-P-N}]
We will construct the sequences $\bp=p_1, \ldots, p_{|P|}$ and $\bn=n_1, \ldots, n_{|N|}$ from the larger indices to the smaller indices.  Suppose that we have already chosen the values of $p_{|P|},p_{|P|-1}, \ldots, p_{k+1}$ and $n_{|N|}, n_{|N|-1}, \ldots, n_{\ell+1}$.  At the next step, we will choose the value of either $p_k$ or $n_\ell$ depending on the sign of the sum of all of the remaining elements.  Let $$P_k\vcentcolon=P \setminus \{p_{|P|}, \ldots, p_{k+1}\} \quad \text{and} \quad N_\ell\vcentcolon=N \setminus \{n_{|N|}, \ldots, n_{\ell+1}\}$$ be the sets of remaining elements of $P$ and $N$, and define the quantities
$$\pi_k\vcentcolon=\sum_{p \in P_k}p  \quad \text{and} \quad \nu_\ell\vcentcolon=\sum_{n \in N_\ell}n.$$
As in the algorithm from \cite{kravitz}, we will succeed in constructing a two-sided valid ordering as long as $p_k,n_\ell$
avoid a few particular potential values (for more details, see Claim \ref{claim:valid} below).  If we are choosing $p_k$, then we want $p_k$ not to be equal to $\pi_k+\delta+\nu_\ell$, since this choice of $p_k$ would lead to a zero-sum interval $p_{k-1}+\cdots+p_1+\delta+n_1+\cdots+n_\ell=0$.  Likewise, if we are choosing $n_\ell$, then we want $n_\ell$ not to be equal to either $\delta+\nu_\ell$ or $\pi_k+\delta+\nu_\ell$, since these choices of $n_\ell$ would lead to zero-sum intervals $\delta+n_1+\cdots+n_{\ell-1}=0$ and $p_k+\cdots+p_1+\delta+n_1+\cdots+n_{\ell-1}=0$.  With this in mind, we define the sets
$$P'_k\vcentcolon=P_k \setminus \{\pi_k+\delta+\nu_\ell\} \quad \text{and} \quad N'_\ell\vcentcolon=N_\ell \setminus \{\delta+\nu_\ell,\pi_k+\delta+\nu_\ell\}$$
of ``allowable'' choices for $p_k$ and $n_\ell$.  Due to the assumptions in Proposition \ref{prop:ordering-P-N}, it will always transpire that the set $P'_k$ or $N'_\ell$ under consideration is nonempty.

Again as in the algorithm from \cite{kravitz}, the sign of the quantity $\pi_k+\delta+\nu_\ell$ will determine whether we choose the value of $p_k$ or the value of $n_\ell$ next:\footnote{The apparent asymmetry in the cases arises from the assumption in Proposition \ref{prop:ordering-P-N} that $\delta>0$.}

\begin{enumerate}
    \item Suppose that $\pi_k+\delta+\nu_\ell \geqslant 0$ and $k>0$.  Then we will choose $p_k\in P'_k$ as follows. If there is some $p^* \in P'_k$ such that $\pi_k-p^* \notin \cup_j Y^+_j$, then choose $p_k$ to be this $p^*$ and say that the current step is a \emph{skip-step for $P$}.
    
    Now, consider the case where $\pi_k-P'_k \subseteq \cup_j Y^+_j$.  Let $i$ be minimal such that $\pi_k-P'_k$ intersects $Y^+_i$, and say that the current step is an \emph{$i$-step for $P$}.  If $\pi_k-P'_k\subseteq Y^+_i$, then let $p_k$ be the largest element of $P'_k$.  If $\pi_k-P'_k \not\subseteq Y^+_i$, then let $p_k$ be the largest $p^* \in P'_k$ such that $\pi_k-p^* \notin Y^+_i$.

    \item Suppose that $\pi_k+\delta+\nu_\ell \geqslant 0$, $k=0$ (i.e., we have already chosen all of $\bp$), and $\ell>0$, or that $\pi_k+\delta+\nu_\ell< 0$ and $\ell>0$.  Then we will choose $n_\ell \in N'_\ell$ as follows.  If there is some $n^* \in N'_\ell$ such that $\nu_\ell-n^* \notin \cup_j Y^-_j$, then choose $n_\ell$ to be this $n^*$ and say that the current step is a \emph{skip-step for $N$}.
    
    Now, consider the case where $\nu_\ell-N'_\ell \subseteq \cup_j Y^-_j$.  Let $i$ be minimal such that $\nu_\ell-N'_\ell$ intersects $Y^-_i$, and say that the current step is an \emph{$i$-step for $N$}.  If $\nu_\ell-N'_\ell \subseteq Y^-_i$, then let $n_\ell$ be the smallest (i.e., most negative) element of $N'_\ell$.  If $\nu_\ell-N'_\ell \not\subseteq Y^-_i$, then let $n_\ell$ be the smallest (i.e., most negative) $n^* \in N'_\ell$ such that $\nu_\ell-n^* \notin Y^-_i$.
\end{enumerate}

We begin with $(k,\ell)=(|P|,|N|)$ and run the above procedure until we reach $(k,\ell)=(0,0)$.  To establish Proposition \ref{prop:ordering-P-N}, we must show three things: that the algorithm actually runs and produces orderings $\bp$ of $P$ and $\bn$ of $N$; that the resulting ordering $\bp,\delta, \overline{\bn}$ is two-sided valid; and that $\IS(\bp)$ and $\IS(\bn)$ have small intersections with the $Y^+_j$'s and $Y^-_j$'s (respectively).

\begin{claim}
The above algorithm runs all the way to $(k,\ell)=(0,0)$ and produces orderings $\bp$ of $P$ and $\bn$ of $N$.   
\end{claim}

\begin{proof}
Note that as long as $(k,\ell) \neq (0,0)$, we fall into one of the two cases.  Indeed, when $\pi_k+\delta+\nu_\ell \geqslant 0$, we fall into case (1) or case (2) according to whether $k>1$ or $k=0$.  When $\pi_k+\delta+\nu_\ell < 0$, we must have $\ell>0$ since $\delta>0$ and $\pi_k \geqslant 0$, so we fall into case (2).

It remains to show that the sets $P'_k, N'_\ell$ are always nonempty when needed.  First, consider case (1).  It is clear that $P'_k \neq \emptyset$ as long as $k>1$.  When $k=1$, the set $P_1$ consists of a single element $p_1$, and we have $\pi_1=p_1$.  We must show that $\pi_1+\delta+\nu_\ell \neq p_1$, i.e., that $\delta \neq -\nu_\ell$.  When $\ell=|N|$, this is precisely the assumption in Proposition \ref{prop:ordering-P-N} that $\delta \neq -\sum_{n \in N}n$.  For $\ell<|N|$, recall that $n_{\ell+1}$ was chosen to be an element of $N'_{\ell+1}$, which by construction does not contain $\delta+\nu_{\ell+1}$.  It follows that
$\nu_\ell=\nu_{\ell+1}-n_{\ell+1} \neq \nu_{\ell+1}-(\delta+\nu_{\ell+1})=-\delta$, as desired.

Now, consider case (2).  We begin with the subcase where $\pi_k+\delta+\nu_\ell \geqslant 0$ and $k=0$.  Note that $\pi_0=0$ and hence $\pi_0+\delta+\nu_\ell=\delta+\nu_\ell$.  It is clear that $N'_\ell \neq \emptyset$ as long as $\ell>1$.  When $\ell=1$, the set $N_1$ consists of a single element $n_1$, and we have $\nu_1=n_1$.  The assumption $\delta>0$ ensures that $\delta+\nu_1=\delta+n_1 \neq n_1$, so $N'_1 \neq \emptyset$.

Finally, we treat the subcase where $\pi_k+\delta+\nu_\ell<0$.  It is clear that $N'_\ell \neq \emptyset$ as long as $\ell>2$.  When $\ell=2$, the set $N_2$ consists of two elements $n_1,n_2$, and we have $\nu_2=n_1+n_2$.  Then $\delta+\nu_2=\delta+n_1+n_2 \notin N_2$ by the assumption in Proposition \ref{prop:ordering-P-N} that $\delta\notin -N$ so neither of $n_1,n_2$ is equal to $-\delta$, and it follows that $N'_2 \neq \emptyset$.  When $\ell=1$, the set $N_1$ consists of a single element $n_1$, and we have $\nu_1=n_1$.  Then $\delta+\nu_1=\delta+n_1 \neq n_1$ since $\delta>0$, and $\pi_k+\delta+\nu_1=\pi_k+\delta+n_1 \neq n_1$ since $\pi_k+\delta>0$.  Thus $N'_1 \neq \emptyset$, and this concludes the proof.
\end{proof}

\begin{claim}\label{claim:valid}
The ordering $\overline \bp,\delta, \bn$ is two-sided valid.
\end{claim}

\begin{proof}
Since any zero-sum interval must contain both positive and negative numbers, we can restrict our attention to intervals of the form $p_k+\cdots+p_1+\delta+n_1+\cdots+n_\ell$ (with sum $\pi_k+\delta+\nu_\ell$) and $\delta+n_1+\cdots+n_\ell$ (with sum $\delta+\nu_\ell$).

Let us first consider the sums $\pi_k+\delta+\nu_\ell$.  Note that we do not need to worry about $(k,\ell)=(|P|,|N|)$, since the corresponding interval is the entire sequence $\overline \bp,\delta,\bn$, so we may assume that either $k<|P|$ or $\ell<|N|$.  Let $(k^*,\ell^*)$ be the earliest step in the algorithm where $k^* \leqslant k$ and $\ell^* \leqslant \ell$.  Then the previous step in the algorithm was either $(k^*+1,\ell^*)$ or $(k^*, \ell^*+1)$; without loss of generality assume that it was the former, since the argument for the latter is identical.  Then $k^*=k$ and $\ell^* \leqslant \ell$.  If $\ell^*=\ell$, then $$\pi_k+\delta+\nu_\ell=(\pi_{k+1}-p_{k+1})+\delta+\nu_\ell$$ is nonzero because we chose $p_{k+1}\in P'_{k+1}$ and the set $P'_{k+1}$ does not contain $\pi_{k+1}+\delta+\nu_\ell$.  If instead $\ell^*<\ell$, then there is some $k'>k$ such that $n_\ell$ was chosen at step $(k',\ell)$. 
 It follows that
$$\pi_k+\delta+\nu_\ell<\pi_{k'}+\delta+\nu_\ell<0,$$
so $\pi_k+\delta+\nu_\ell$ is nonzero, as desired.

Let us now consider the sums $\delta+\nu_\ell$.  We can again quickly dispose of the case $\ell=|N|$.  Indeed, if $P=\emptyset$, then the corresponding interval is the entire sequence $\overline\bp,\delta,\bn$, and if $P \neq \emptyset$, then $\delta \neq -\sum_{n \in N}n=-\nu_{|N|}$ by assumption.  So we may assume that $\ell<|N|$, and we conclude by noting that $\delta+\nu_\ell=\delta+(\nu_{\ell+1}-n_{\ell+1}) \neq 0$ since $N'_{\ell+1}$ does not contain $\delta+\nu_{\ell+1}$.
\end{proof}

\begin{claim}
The orderings $\bp$ and $\bn$ satisfy
$$|\IS(\bp) \cap Y^+_j| \leqslant \inf_{L \in \mathbb{N}} \left(\frac{|Y^+_j|}{L}+L+2+4\sum_{i=1}^{j-1}|Y^+_i|\right)$$ and $$|\IS(\bn) \cap Y^-_j| \leqslant \inf_{L \in \mathbb{N}} \left(\frac{|Y^-_j|}{L}+L+2+4\sum_{i=1}^{j-1}|Y^-_i|\right)$$
for all $1 \leqslant j \leqslant m$.
\end{claim}

\begin{proof}
We will prove the statement only for $|\IS(\bn) \cap Y^-_j|$ since the argument for $|\IS(\bp) \cap Y^+_j|$ is essentially identical.  Recall that $\IS(\bn)=\{\nu_\ell=\sum_{i=1}^\ell n_i: 0 \leqslant \ell \leqslant |N|\}$.  For $0 \leqslant \ell<|N|$, write $\nu_\ell=\nu_{\ell+1}-n_{\ell+1}$.  This quantity can lie in $Y^-_j$ only when the choice of $n_{\ell+1}$ is a $j$-step or an $i$-step for some $i<j$.  We will bound these two contributions separately.  Note that skip-steps and $i$-steps for $i>j$ never contribute.

We first consider the contribution of $j$-steps.  Notice that the partial sums $\nu_\ell$ are strictly increasing (becoming less negative) as $\ell$ decreases.  Suppose that the choice of $n_{\ell+1}$ is a $j$-step and $\nu_\ell=\nu_{\ell+1}-n_{\ell+1} \in Y^-_j$.  Then we must have $\nu_{\ell+1}-N'_{\ell+1} \subseteq Y^-_j$.  Since $n_{\ell+1}$ is the smallest (most negative) element of $N'_{\ell+1}$, the other $|N'_{\ell+1} \setminus \{n_{\ell+1}\}| \geqslant \ell+1-3=\ell-2$ elements of $\nu_{\ell+1}-N'_{\ell+1}\subseteq Y^-_j$ lie in the interval $(\nu_{\ell+1},\nu_\ell)$; it follows that these elements are ``skipped'' and can never appear in $\IS(\bn)$.  In particular, from such $j$-steps with $\ell\geqslant L+1$ we obtain at most $|Y^-_j|/L$ elements of $\IS(\bn) \cap Y^-_j$.  From $j$-steps with $\ell \leqslant L$ we trivially obtain at most $L+1$ elements of $\IS(\bn) \cap Y^-_j$.

We now consider the contribution of $i$-steps with $i<j$.  We will trivially bound this contribution by the total number of $i$-steps with $i<j$.  We claim that the number of $i$-steps is at most $4|Y^-_i|$ for each $i$.  For each $i$-step $\ell$, let $y(\ell)$ denote the largest (least negative) element of $(\nu_\ell-N'_\ell) \cap Y^-_i$.  It suffices to show that each $y \in Y^-_i$ appears as $y(\ell)$ for at most $4$ different $i$-steps $\ell$.  If $y(\ell)$ is not the largest element of $\nu_\ell-N'_\ell$, then it is distinct from $y(\ell')$ for all $\ell'<\ell$ since $$\nu_{\ell'} \geqslant\nu_{\ell-1}=\nu_\ell-n_\ell>y(\ell)$$ by the definition of an $i$-step.  If $y(\ell)$ is the largest element of $\nu_\ell-N'_\ell$, then it is one of the three largest elements of $\nu_\ell-N_\ell$.  Notice that the largest element of $\nu_\ell-N_\ell$ is strictly increasing as $\ell$ decreases, the second-largest element of $\nu_\ell-N_\ell$ is strictly increasing as $\ell$ decreases, and the third-largest element of $\nu_\ell-N_\ell$ is strictly increasing as $\ell$ decreases; it follows that each $y$ can appear at most three times as one of the three largest elements of $\nu_\ell-N_\ell$.  Thus we have shown that each $i$-step $\ell$ in the algorithm is associated with some number $y(\ell)\in Y^-_i$ and moreover that any given $y \in Y^-_i$ appears as $y(\ell)$ for at most $4$ different $i$-steps $\ell$, so we conclude that the total number of $i$-steps is at most $4|Y^-_i|$.  This establishes the claim.

Combining these contributions (and adding $1$ for $\nu_{|N|}$) gives the desired upper bound.\footnote{To bound the intersection between $\IS(\bp)$ and $Y^+_j$, one simply interchanges ``smaller'' and ``larger'' throughout the proof.  Since we have $|P_k \setminus P'_k| \leq 1$ instead of $|N_\ell \setminus N'_\ell| \leq 2$, we could replace $\ell-2$ with $k-1$ in the second paragraph and replace $4|Y_j^-|$ with $3|Y_j^+|$ in the third paragraph to obtain even a slightly tighter bound.}
\end{proof}

These three claims together imply Proposition \ref{prop:ordering-P-N}.
\end{proof}

\section{Splitting the dissociated sets}
\label{sec:splitting}
In this section we manipulate the dissociated sets $D_j$ in order to make their sums suitably generic; this will avoid ``bad'' scenarios in the random orderings of the $D_j$'s that we will consider in the next section. 
Recall that if $D$ is a dissociated set, then all of the subset sums of $D$ are distinct. In particular, if we choose a uniformly random partition of $D$ into parts $D^{(1)}, D^{(2)},D^{(3)},D^{(4)}$ of equal size (up to rounding), 
then (omitting floor functions) for each $1 \leqslant i \leqslant 4$ the $\binom{|D|}{|D|/4}$ possible values of $\sum_{d \in D^{(i)}}d$ are all achieved with equal probability; likewise, each of the quantities $\sum_{d \in D^{(1)} \cup D^{(2)}}d,\sum_{d \in D^{(2)} \cup D^{(3)}}d,\sum_{d \in D^{(3)} \cup D^{(4)}}$ is uniformly distributed on $\binom{|D|}{|D|/2}$ possible values, and each of the quantities $\sum_{d \in D^{(1)} \cup D^{(2)} \cup D^{(3)}}d,\sum_{d \in D^{(2)} \cup D^{(3)} \cup D^{(4)}}d$ is uniformly distributed on $\binom{|D|}{3|D|/4}$ possible values.  Since $\binom{|D|}{|D|/4},\binom{|D|}{|D|/2},\binom{|D|}{3|D|/4}$ are all $e^{\Omega(|D|)}$, we obtain very strong anti-concentration for the sums under consideration.  We record this simple but important fact in the following lemma.

\begin{lemma}\label{splittinglemma}
   Let $D\subset G$ be a dissociated set, and let $D=D^{(1)} \cup D^{(2)} \cup D^{(3)} \cup D^{(4)}$ be a uniformly random partition of $D$ into four sets of equal size (up to rounding).  Then for every nonempty proper interval $I \subseteq [4]$ and every $x\in G$, we have $$\mathbb{P}\left(\sum_{i \in I}\sum_{d \in D^{(i)}}d=x\right)\leqslant e^{-\Omega(|D|)}.$$
\end{lemma}

Let $D_1,\dots,D_{s}$ be the dissociated sets appearing in the structural decomposition of $A$ from Proposition \ref{prop:refined-structure}.  We will split and reorder these dissociated sets as follows.  For each $j\in[1,s]$, we partition $D_j=\cup_{i=1}^4 D_j^{(i)}$ into four sets of equal size uniformly at random as in Lemma \ref{splittinglemma}, and we require that $|D_1^{(1)}|=|D_s^{(4)}|$.  We do all of these splittings independently. Next, we place these newly formed dissociated sets in the order
\begin{align}\label{eq:splitandrearrange}
    D_1^{(1)},D_1^{(2)},D_2^{(1)},D_2^{(2)},\ldots,D_{s}^{(1)},D_s^{(2)},D_1^{(3)},D_1^{(4)},D_2^{(3)},D_2^{(4)},\ldots,D_s^{(3)},D_{s}^{(4)}
\end{align}
and note that of course the decomposition $$\lambda\cdot A = P\cup N\cup (\cup_{i=1}^4\cup_{j=1}^{s}D_j^{(i)})$$ still holds (with the same value of $\delta$).  For notational convenience, write $T_1,T_2,\dots,T_u$ (with $u=4s$) for the new sequence of dissociated sets in \eqref{eq:splitandrearrange}, and let $\tau_j\vcentcolon=\sum_{t\in T_j}t$.

Let us pause at this point and describe the remainder of the strategy for proving Theorem \ref{thm:main}.  We will eventually construct a two-sided valid ordering of $A$ of the form $$\overline{\bp},\mathbf{t}_1,\dots,\mathbf{t}_u,\bn,$$ where each $\mathbf{t}_i$ is an ordering of $T_i$ chosen randomly according to a certain distribution. Our task will be to show that such an ordering $a_1,\dots,a_{|A|}$ is likely to avoid zero-sum subintervals, namely, proper nonempty intervals $ I\subset[|A|]$ with $\sum_{i\in I}a_i=0$.  For the remainder of the paper, we will refer to proper nonempty intervals as simply ``intervals''.  We divide such intervals $I$ into two ``types'', which we will treat using different arguments.  Recall that $K=c_2 R^{1/3}$.

\begin{definition}
    Let $I\subset [|A|]$ be a proper nonempty interval. We say that $I$ is \emph{Type II} if it contains between $K$ and $|T_j|-K$ elements of some $T_j$, and otherwise we say that it is \emph{Type I}.
\end{definition}
We will refer to the first $K$ elements in an ordering $\bt=t_1,\dots,t_m$ as its left border and to the final $K$ elements as its right border. The remaining elements $t_{K+1},\dots,t_{m-K}$ make up the interior region of $\bt$.  In this language (and ignoring intervals contained in a single $T_j$, which can never be zero-sum), a Type II interval  is an interval with at least one endpoint in the interior region of one of the orderings $\bt_j$, and a Type I interval is an interval with each endpoint in $\overline{\bp}$, $\bn$, or a border region of some $\bt_j$.
One should think of Type II intervals as generic and of Type I intervals as exceptional.  (Obviously the identification of intervals $I\subset [|A|]$ as Type I and Type II does not depend on the random choices of the $\bt_j$'s).

The main benefit of the above splitting-and-rearranging procedure is that it lets us dispose of nearly all Type I intervals even before we choose the random orderings $\bt_j$.  The following lemma makes this precise.  We say that an event holds \emph{with high probability} if it holds with probability tending to $1$ as $p$ tends to infinity.


\begin{lemma}\label{splitandrearrange}
Let $c>0$ be any constant.  Let $1 \leqslant s \leqslant e^{c(\log p)^{1/4}}$, and let $D_1, \ldots, D_s \subseteq \mathbb{F}_p$ be dissociated sets each of size $\asymp R$, with the property that $D_1 \cup D_s \cup\{\delta\}$ is dissociated.  Let $\bp$ and $\bn$ be sequences over $\mathbb{F}_p$ each of length at most $e^{c(\log p)^{1/4}}$, and assume that $\overline{\bp},\delta,\bn$ is a two-sided valid ordering.  If the sequence $T_1, \ldots, T_u$ of dissociated sets is chosen randomly as described above, then each $|T_j|\asymp R$ and each $T_{2j-1}\cup T_{2j}$ is dissociated, and the following holds with high probability:
\begin{enumerate}[(i)]
    \item for each proper nonempty interval $I=[i,j] \subseteq [u]$, we have that 
    $$0 \notin \left(\sum_{\leqslant K}(T_{i-1}) \cup -\sum_{\leqslant K}(T_i) \right)+\tau_i+\cdots+\tau_j+ \left(-\sum_{\leqslant K}(T_j) \cup \sum_{ \leqslant K}(T_{j+1}) \right)$$
    (with the convention that $T_0=T_{u+1}=\emptyset$);
    \item for each $1 \leqslant j \leqslant u-1$, we have that
    $$0\notin \IS(\bp)+\tau_1+\cdots+\tau_j+ \left(-\sum_{\leqslant K}(T_j) \cup \sum_{ \leqslant K}(T_{j+1}) \right);$$
    and for each $2 \leqslant j \leqslant u$, we have that
    $$0 \notin \IS(\bn)+\tau_u+\cdots+\tau_j+ \left(-\sum_{\leqslant K}(T_j) \cup \sum_{ \leqslant K}(T_{j-1}) \right);$$
    \item the ordering $\overline \bp,\tau_1\dots,\tau_{u},\mathbf{n}$ is two-sided valid.
\end{enumerate}
\end{lemma}

Three remarks are in order before we proceed to the proof.
\begin{enumerate}
    \item To see how this lemma pertains to Type I intervals containing nearly all (i.e., at least $|T_j|-K$ elements) of some $T_j$, simply note the identity
    $\sum_{ \geqslant|T_j|-K}(T_j)=\tau_j-\sum_{ \leqslant K}(T_j)$.
    \item Items (i)--(iii) handle all Type I intervals except for the following:
    \begin{itemize}
        \item intervals fully contained in a single $T_j$;
        \item intervals starting in the left border of $\bt_1$ and ending in the right border of $\bt_u$;
        \item intervals beginning in the right border of $\bt_j$ and ending in the left border of $\bt_{j+1}$ for some $j$;
        \item intervals with one endpoint in $\overline{\bp}$ or $\bn$ and the other endpoint in the left border of $\bt_1$ or the right border of $\bt_u$.
    \end{itemize}
    Moreover, the first case cannot lead to zero-sum intervals because each $T_j$ is dissociated; likewise, there cannot be zero-sum intervals in the second case because of the assumption that $D_1 \cup D_s \cup \{\delta\}$ (and a fortiori $T_1 \cup T_u \cup \{\delta\}$) is dissociated.
    In the third case, we never have to worry about zero-sum intervals with $j$ odd since each $T_{2k-1} \cup T_{2k}$ is dissociated.
    \item The lemma would continue to hold with $K$ as large as a small constant times $R$, but we will not have occasion to make use of this fact.
\end{enumerate}

\begin{proof}[Proof of Lemma \ref{splitandrearrange}]
    We begin with the crucial observation that if $I\subset [u]$ is any proper nonempty subinterval and $x \in \mathbb{F}_p$ is any element, then we have the anti-concentration inequality
    $$\mathbb{P}\left(\sum_{i\in I}\tau_i=x\right)=e^{-\Omega(R)}.$$
    Indeed, there is some $j \in [s]$ such that $\{T_i:i\in I\}$ contains at least one but not all of $D_j^{(1)},\dots,D_j^{(4)}$.  Suppose that it contains $D_j^{(1)}$ but none of $D_j^{(2)},D_j^{(3)},D_j^{(4)}$ (the remaining cases are analogous).  Since the splitting of $D_j$ is independent of the splittings of the other $D_k$'s, Lemma \ref{splittinglemma} gives
\begin{align*}
\mathbb{P}\left(\sum_{i\in I}\tau_i=x\right) &=\sum_{z \in \mathbb{F}_p} \mathbb{P} \left(\sum_{i \in I\setminus \{2j-1\}}\tau_i =z \quad \text{and} \quad \tau_{2j-1}=x-z \right)\\
 &=\sum_{z \in \mathbb{F}_p} \mathbb{P} \left(\sum_{i \in I\setminus \{2j-1\}} \tau_i=z \right) \mathbb{P} \left(\sum_{d\in D_j^{(1)}}d=x-z \right)\\
 & \leqslant \sum_{z \in \mathbb{F}_p} \mathbb{P} \left(\sum_{i \in I\setminus \{2j-1\}} \tau_i=z \right) e^{-\Omega(|D_j|)}=e^{-\Omega(R)}.
\end{align*}    
With this observation in hand, we proceed to the main body of the proof.  Note that (iii) holds whenever (i) and (ii) hold since $0 \in \sum_{ \leqslant K}(T_j)$ and as we assumed that $\overline{\bp},\delta, \bn$ is two-sided valid. So, by the union bound, it suffices to show that each  of (i) and (ii) holds with high probability.

We begin with (i).  Fix some proper nonempty interval $I=[i,j] \subseteq [u]$.  The assertion of (i) for this $I$ is that
$$\tau_i+\cdots+\tau_j \notin \left(-\sum_{\leqslant K}(T_{i-1}) \cup \sum_{\leqslant K}(T_i) \right)+\left(\sum_{\leqslant K}(T_j) \cup -\sum_{ \leqslant K}(T_{j+1}) \right).$$
The set on the right-hand side has size at most
$$\left(\left|\sum_{\leqslant K}(T_{i-1}) \right|+\left|\sum_{\leqslant K}(T_i)\right|\right) \left(\left|\sum_{\leqslant K}(T_j) \right|+\left|\sum_{\leqslant K}(T_{j+1})\right|\right) \leqslant e^{O(R \cdot  H(O(K/R)))},$$
where $H(x)\vcentcolon=-x\log_2(x)-(1-x)\log_2(1-x)$ is the binary entropy function.  The definition of $K$ ensures that $H(O(K/R))=o(1)$ (with a lot of room to spare), and then the observation from the beginning of the proof tells us that (i) fails for $I$ with probability at most $e^{o(R)-\Omega(R)}=e^{-\Omega(R)}$.
A union bound over the (at most $u^2$) choices of $I$ shows that (i) fails with probability at most
$$u^2 e^{-\Omega(R)}\leqslant e^{2c (\log p)^{1/4}-\Omega(c_1 (\log p)^{3/4})}=o(1),$$
again with plenty of room to spare.

The proof of (ii) is nearly identical and we omit it; we remark that the bounds $|\IS(\bp)|, |\IS(\bn)| \leqslant |A|\leqslant e^{c (\log p)^{1/4}}$ hold because we fixed $\bp$ and $\bn$ in advance.

\end{proof}

As noted in remark (2) following Lemma \ref{splitandrearrange}, there remain two sorts of Type I intervals to address. The first is Type I intervals contained in $T_{2k}\cup T_{2k+1}$ for some $k$. We can avoid zero-sums here by picking the orderings $\bt_{2k},\bt_{2k+1}$ according to a suitable joint distribution which we will describe in section \ref{sec:random}. The second is Type I intervals with one endpoint in $\overline{\bp}$ or $\bn$ and the other endpoint in the left border of $\bt_1$ or the right border of $\bt_u$. The crucial ingredient for dealing with these will turn out to be the last part of Proposition \ref{prop:ordering-P-N}.  Since Proposition \ref{prop:ordering-P-N} must be applied prior to the random splitting procedure described in this section, it is a bit of a nuisance that the input sets $Y_j^+,Y_j^-$ must be described in terms of the sets $D_j$ rather than the sets $T_j$.  The following lemma will let us remedy this issue.

\begin{lemma}\label{lemm:fewcollisions}
    Let $T_1=D_1^{(1)}$ and $T_u=D_s^{(4)}$ be the random sets from \eqref{eq:splitandrearrange}. Then with probability at least $1/2$, we have for all $1\leqslant j\leqslant K$ that
    \begin{align*}
        \left|\sum_{=j}(T_1)\cap (-\operatorname{IS}(\bp)\cup(\delta+\operatorname{IS}(\bn)))\right|\leqslant 4K\frac{{|T_1|\choose j}}{{|D_1|\choose j}}\left|\sum_{=j}(D_1)\cap (-\operatorname{IS}(\bp)\cup(\delta+\operatorname{IS}(\bn)))\right|,\\
        \left|\sum_{=j}(T_u)\cap (-\operatorname{IS}(\bp)\cup(\delta+\operatorname{IS}(\bn)))\right|\leqslant 4K\frac{{|T_u|\choose j}}{{|D_s|\choose j}}\left|\sum_{=j}(D_s)\cap (-\operatorname{IS}(\bp)\cup(\delta+\operatorname{IS}(\bn)))\right|.
    \end{align*}
\end{lemma}
\begin{proof}
    Since $D_1$ is dissociated, the quantity $\left|\sum_{=j}(D_1)\cap (-\operatorname{IS}(\bp)\cup(\delta+\operatorname{IS}(\bn)))\right|$ simply counts the subsets $S\subseteq D_1$ of size $|S|=j$ with $\sum_{d\in S}d\in -\operatorname{IS}(\bp)\cup(\delta+\operatorname{IS}(\bn))$, and likewise for $T_1$. As $T_1=D_1^{(1)}$ is chosen uniformly from all subsets of $D_1$ of size $|D_1|/4$, we have
    \begin{align*}
        \mathbb{E}\left(\left|\sum_{=j}(T_1)\cap (-\operatorname{IS}(\bp)\cup(\delta+\operatorname{IS}(\bn)))\right|\right)=\frac{{|T_1|\choose j}}{{|D_1|\choose j}}\left|\sum_{=j}(D_1)\cap (-\operatorname{IS}(\bp)\cup(\delta+\operatorname{IS}(\bn)))\right|,
    \end{align*}
    and Markov's Inequality implies that the first bound in the conclusion of the lemma fails for each $j$ with probability at most $1/4K$.  The same argument applies with $D_s,T_u$ in place of $D_1,T_1$, and the conclusion of the lemma follows from a union bound over $1 \leqslant j \leqslant K$.
\end{proof}

\section{Randomizing the dissociated sets}\label{sec:random}

We are finally ready to describe how we will construct a two-sided valid ordering of $A$. Suppose that $A\subseteq \mathbb{F}_p \setminus \{0\}$ has size $|A|\leqslant e^{c(\log p)^{1/4}}$.  After we replace $A$ by a suitable dilate (which is harmless with regard to finding two-sided valid orderings), Proposition \ref{prop:refined-structure} provides a decomposition
\begin{align*}
    A=P\cup N\cup(\cup_{j=1}^s D_j)
\end{align*} satisfying conditions (i)-(v) of that proposition and $\delta=\sum_j\sum_{d\in D_j}d>0$. Now, with $K=c_2R^{1/3}$ for a suitably small constant $c_2>0$, set
\begin{align}\label{Y_jdefinition}
    Y^+_j\vcentcolon=-\sum_{= j}(D_1)\cup\left(-\delta+\sum_{= j}(D_s)\right) \quad \text{and} \quad Y^-_j\vcentcolon=-\sum_{= j}(D_s)\cup\left(-\delta+\sum_{= j}(D_1)\right)
\end{align} for each $1\leqslant j \leqslant K$, and apply Proposition \ref{prop:ordering-P-N}. This provides orderings $\bp$ of $P$ and $\bn$ of $N$ such that the sequence $\overline{\bp},\delta,\bn$ is two-sided valid and such that \eqref{ISp-bound},\eqref{ISn-bound} hold.  Finally, we can use Lemmas \ref{splitandrearrange} and \ref{lemm:fewcollisions} to obtain dissociated sets $T_1. \ldots, T_u$ from $D_1, \ldots, D_s$ such that $A=P\cup N\cup(\cup_i T_i)$ and the conclusions of these two lemmas are simultaneously satisfied; fix such a choice of $T_1, \ldots, T_u$. Recall that $\tau_i=\sum_{t\in T_i}t$, and write $m_i:=|T_i|$. The two-sided valid ordering of $A$ that we will construct will be of the form $$\overline{\bp},\bt_1,\dots,\bt_u,\bn,$$ where the $\mathbf{t}_i$'s are orderings of the $T_i$'s chosen randomly according to certain distributions, whose description and analysis occupies the remainder of this section.

Recall that if $T$ is a dissociated set, then all of the subset sums of $T$ are distinct.  In particular, in a uniformly random ordering of the elements of $T$, the sum of the first $k$ elements is uniformly distributed on $\binom{|T|}{k}$ values.  As long as $k$ is not too close to $0$ or $|T|$, this sum is very anti-concentrated, and so with very high probability it will avoid any particular small set of values.  It follows that uniformly random orderings $\bt_i$ would with high probability avoid zero-sum Type II intervals.  We can ignore most Type I intervals due to Lemma \ref{splitandrearrange}, but the remaining Type I intervals, as described in remark (2) following that lemma, still cause issues.  We will show that each of these potential zero-sum Type I intervals can be avoided ``locally'' by introducing some non-uniformity into the distributions determining the orderings $\bt_i$.

We begin with the orderings $\bt_1,\bt_u$.  Say that an ordering $t_1,\dots,t_{m_1}$ of $T_1$ is \emph{acceptable} if $$t_1+\dots+t_k\notin -\IS(\bp)\cup \left(\delta+\IS(\bn)\right) \quad \text{for all $1 \leqslant k \leqslant K$},$$ and say that an ordering $t_1,\dots,t_{m_u}$ of $T_u$ is \emph{acceptable} if $$t_1+\dots+t_k\notin -\IS(\bn)\cup \left(\delta+\IS(\bp)\right) \quad \text{for all $1 \leqslant k \leqslant K$}.$$
Using Proposition \ref{prop:ordering-P-N} and the fact that $T_1,T_u$ satisfy the conclusion of Lemma \ref{lemm:fewcollisions}, we can show that uniformly random orderings of $T_1,T_u$ are acceptable with large probability.
\begin{lemma}\label{lemm:permissible}
    A uniformly random ordering of $T_1$ is acceptable with probability at least $0.98$, and a uniformly random ordering of $T_u$ is acceptable with probability at least $0.98$. 
\end{lemma}
\begin{proof}
We prove only the statement for $T_1$ since the argument for $T_u$ is identical.  Let $t_1, \ldots, t_{|T_1|}$ be our uniformly random ordering of $T_1$.  By the union bound, it suffices to show that $\mathbb{P}(t_1 \in -\IS(\bp) \cup (\delta+\IS(\bn))) \leqslant 0.01$ and that
$$\mathbb{P}(t_1+\cdots+t_k \in -\IS(\bp) \cup (\delta+\IS(\bn))) \leqslant 0.01K^{-1}$$
for each $2 \leqslant k \leqslant K$.
Fix some $1 \leqslant k \leqslant K$.  Then the quantity $t_1+\cdots+t_k$ is uniformly distributed on the set $\sum_{=k}(T_1)$, which has size $\binom{|T_1|}{k}$.
Recall that we applied Proposition \ref{prop:ordering-P-N} with the sets $Y_j^{+},Y_j^-$ as in \eqref{Y_jdefinition}.  Since $|D_1|=|D_s|$ and $D_1 \cup D_s \cup \{\delta\}$ is dissociated by Proposition \ref{prop:refined-structure}, we have $|Y_j^+|=|Y_j^-|=2\binom{|D_1|}{j}$ for all $j$.  Then the conclusion of Proposition \ref{prop:ordering-P-N}, with $L:=\lfloor |Y_k^+|^{1/2}\rfloor$, gives
$$\left|\sum_{=k}(D_1) \cap (-\IS(\bp) \cup (\delta+\IS(\bn))) \right| \ll |Y_k^+|^{1/2}+1+\sum_{j<k} |Y_j^+|.$$
For $k=1$, this gives (recall that $|D_1|\asymp R$)
$$\left|\sum_{=1}(D_1) \cap (-\IS(\bp) \cup (\delta+\IS(\bn))) \right| \ll |Y_1^+|^{1/2}\ll\binom{|D_1|}{1} \cdot R^{-1/2},$$
and for $2 \leqslant k \leqslant K$ (recall that $K=c_2 R^{1/3}$) it gives
$$\left|\sum_{=k}(D_1) \cap (-\IS(\bp) \cup (\delta+\IS(\bn))) \right| \ll \binom{|D_1|}{k} \cdot \frac{K}{|D_1|} \ll \binom{|D_1|}{k} \cdot c_2^3 K^{-2}.$$

Since the conclusion of Lemma \ref{lemm:fewcollisions} also holds, we can ``transfer'' this bound from $D_1$ to $T_1$.  In particular, we obtain that
$$\left|\sum_{=1}(T_1) \cap (-\IS(\bp) \cup (\delta+\IS(\bn))) \right| \ll 4K\binom{|T_1|}{1} \cdot R^{-1/2}$$
and that
$$\left|\sum_{=k}(T_1) \cap (-\IS(\bp) \cup (\delta+\IS(\bn))) \right| \ll 4K\binom{|T_1|}{k} \cdot \frac{K}{|D_1|} \ll 4K \binom{|T_1|}{k} \cdot c_2^3 K^{-2}$$
for $2 \leqslant k \leqslant K$.
It follows that
$$\mathbb{P}(t_1 \in -\IS(\bp) \cup (\delta+\IS(\bn))) \ll R^{-1/6}$$
is certainly at most $0.01$, and for $2 \leqslant k \leqslant K$ we see that
$$\mathbb{P}(t_1+\cdots+t_k \in -\IS(\bp) \cup (\delta+\IS(\bn)))\ll c_2^3 K^{-1}$$
is at most $0.01 K^{-1}$ as long as $c_2$ is sufficiently small.  This completes the proof.

\end{proof}
We now choose $\bt_1,\bt_u$ independently such that $\bt_1, \overline{\bt_u}$ are  uniformly random acceptable orderings of $T_1,T_u$, respectively.  We deduce from Lemma \ref{lemm:permissible} that the random variables $\bt_1, \bt_u$ are highly anti-concentrated in the sense that the probability of $\bt_1$ assuming any particular ordering is at most $\ll 1/m_1!$, and likewise for $\bt_u$.  Notice that the constraint that $\bt_1,\overline{\bt_u}$ are acceptable precisely guarantees the absence of zero-sum Type I intervals with one endpoint in $\overline{\bp}$ or $\bn$ and the other endpoint in the left border of $\bt_1$ or the right border of $\bt_u$.

For each $1 \leqslant j \leqslant u/2-1$, we choose the pair of orderings $\bt_{2j}, \bt_{2j+1}$ as follows.  Recall that $|T_{2j}|=m_{2j}$ and $|T_{2j+1}|=m_{2j+1}$ both have size $\asymp R$.  
Say that a pair of partial orderings $t_1, \ldots, t_k$ of $T_{2j}$ and $t'_1, \ldots, t'_\ell$ of $T_{2j+1}$ is \emph{permissible} if
$$t_1+\cdots+t_i + t'_1+\cdots+t'_j\neq 0 \quad \text{for all $(i,j)$}.$$
Let $N(k,\ell)$ denote the number of permissible pairs with lengths $(k,\ell)$.  Note that each permissible pair with lengths $(k,\ell)$ can be extended to at least $m_{2j}-k-\ell$ permissible pairs of lengths $(k+1,\ell)$ and to at least $m_{2j+1}-k-\ell$ permissible pairs of lengths $(k,\ell+1)$.  It follows that
$$N(k,k) \geqslant (m_{2j})(m_{2j+1}-1)(m_{2j}-2)(m_{2j+1}-3) \cdots (m_{2j}-2k+2)(m_{2j+1}-2k+1).$$
The choice of $K$ ensures that
$N(K,K) \geqslant m_{2j}^K m_{2j+1}^K/2$
(say), which means that the permissible pairs comprise at least a constant fraction of the total pairs.  (In fact, this bound would continue to hold with $K$ as large as a small constant times $R^{1/2}$.)  

We now choose $\bt_{2j}, \bt_{2j+1}$ to be a uniformly random pair of orderings of $T_{2j}, T_{2j+1}$ conditional on the length-$K$ prefixes of $\overline{\bt_{2j}}, \bt_{2j+1}$ forming a permissible pair of length $(K,K)$.  Equivalently, we let $t_1, \ldots, t_K$ and $t'_1, \ldots, t'_K$ be a uniformly random permissible pair of orderings of $T_{2j}, T_{2j+1}$, and then we let $\bt_{2j}$ be a uniformly random ordering of $T_{2j}$ conditional on $\overline{\bt_{2j}}$ beginning with $t_1, \ldots, t_K$, and we let $\bt_{2j+1}$ be a uniformly random ordering of $T_{2j+1}$ conditional on $\bt_{2j+1}$ beginning with $t'_1, \ldots, t'_K$.  We make these choices independently for different values of $j$, and independently of the choices of $\bt_1,\bt_u$.  The following lemma shows that even though the random variables $\bt_{2j}, \bt_{2j+1}$ are dependent, they are ``conditionally anti-concentrated'' in the sense that if we condition on $\bt_{2j}$ being any particular ordering, then the the probability of $\bt_{2j+1}$ assuming any particular ordering is still very small, and vice versa.

\begin{lemma}\label{lem:permissible}
Choose $\bt_{2j}, \bt_{2j+1}$ according to the distribution described above.  Then, conditional on $\bt_{2j}$ assuming any particular ordering, the probability of $\bt_{2j+1}$ assuming any particular ordering is $\ll 1/m_{2j+1}!$; likewise, conditional on $\bt_{2j+1}$ assuming any particular ordering, the probability of $\bt_{2j}$ assuming any particular ordering is $\ll 1/m_{2j}!$.
\end{lemma}

\begin{proof}
We prove only the first statement.  Let $\bu_{2j}$ be any fixed ordering of $T_{2j}$.  Let $\bu^{(K)}_{2j}$ denote the ordering consisting of the first $K$ elements of $\bu_{2j}$.  The number of permissible pairs $\bu^{(K)}_{2j} \bu_{2j+1}^{(K)}$ with $\bu_{2j+1}^{(K)}$ of length $K$ is at least
$$(m_{2j+1}-K)^K \gg m_{2j+1}^K,$$
by our choices of $R,K$ (see the above discussion of $N(k,\ell)$).  Thus, the number of orderings $\bu_{2j+1}$ of $T_{2j+1}$ such that $\bu^{(K)}_{2j},\bu^{(K)}_{2j+1}$ is a permissible pair is
$$\gg m_{2j+1}^K \cdot (m_{2j+1}-K)! \gg  m_{2j+1}!.$$
The lemma now follows since each of these $\gg m_{2j+1}!$ orderings of $T_{2j+1}$ is equally likely to occur as $\bt_{2j+1}$, after we condition on $\bt_{2j}=\bu_{2j}$.
\end{proof}

Notice that the constraints on the pairs $\bt_{2j}\bt_{2j+1}$ guarantee the absence of zero-sum Type I intervals beginning in the right border of $\bt_{2j}$ and ending in the left border of $\bt_{2j+1}$.

We will show that if the orderings $\bt_1, \ldots, \bt_u$ of $T_1, \ldots, T_u$ are chosen randomly as above, then with high probability the ordering 
\begin{equation}\label{eq:Arandomordering}
    a_1,\dots,a_{|A|}:=\overline{\bp},\mathbf{t}_1,\dots,\mathbf{t}_u,\bn
\end{equation} of $A$ is two-sided valid, i.e., we have $\sum_{i\in I}a_i\neq 0$ for every nonempty proper interval $I \subseteq [|A|]$.  The output of Lemma \ref{splitandrearrange} and the constraints on the random orderings $\bt_j$ together guarantee that there are no zero-sum Type I intervals in the ordering \eqref{eq:Arandomordering}; the reader can refer to remark (2) following Lemma \ref{splitandrearrange} to see how we have covered all possible cases.  It remains to verify that with high probability there are no zero-sum Type II intervals.
The key point is that the sum $\sum_{i \in I}a_i$ for each Type II interval $I$ is highly anti-concentrated because there is still enough randomness in the orderings $\bt_j$; the following lemma makes this observation precise.
\begin{lemma}\label{lemm:typeII}
    Let $I\subset[1,|A|]$ be a Type II interval, and let $a_1,\dots,a_{|A|}=\overline{\bp},\mathbf{t}_1,\dots,\mathbf{t}_u,\bn$ be the random ordering \eqref{eq:Arandomordering} of $A$. Then for every $x\in \mathbb{F}_p$ we have
    \begin{align*}
        \mathbb{P}\left(\sum_{i\in I}a_i=x\right)\leqslant e^{-\Omega(K \log R)}.
    \end{align*}
\end{lemma}
\begin{proof}
    By definition, there exists some $j$ such that $I$ contains exactly $k$ elements of $T_j$, where $K \leqslant k \leqslant |T_j|-K$.  As in the first step of the proof of Lemma \ref{splitandrearrange}, we break the sum over $I$ into the sum over the part intersecting $T_j$ and the part not intersecting $T_j$ and condition on $\bt_i$ for all $i \neq j$.  Lemma \ref{lem:permissible} ensures that even after this conditioning, the probability of $\bt_j$ assuming any particular ordering is $\ll 1/m_j!$.  Since the $k$-element subsets of $T_j$ all have distinct sums, we see that the sum over the part of $I$ intersecting $T_j$ assumes each particular value with probability $$\ll \binom{m_j}{k}^{-1} \leqslant \binom{m_j}{K}^{-1} \leqslant e^{-\Omega(K \log R)},$$
    and the lemma follows.

\end{proof}

Recall that $K=c_2R^{1/3}$ and that $R=R(A)\gg c_1(\log p)^{3/4}$ holds when $|A|\leqslant e^{c(\log p)^{1/4}}$ (see \eqref{Rdefinition}).  Since the number of Type II intervals is trivially at most $|A|^2 \leqslant e^{2c (\log p)^{1/4}}$, Lemma \ref{lemm:typeII} and the union bound imply that the probability of \eqref{eq:Arandomordering} containing a zero-sum Type II interval is at most $$e^{2c(\log p)^{1/4}-\Omega(K \log R)}=o(1),$$ again with plenty of room to spare.  From this and the above observations about the absence of zero-sum Type I intervals, we conclude that \eqref{eq:Arandomordering} is two-sided valid with high probability; in particular, for $p$ sufficiently large (in terms of $c$), there is at least one two-sided valid ordering of $A$.  This proves Theorem \ref{thm:main}.  

One can in fact take $c$ to grow as, e.g.,  $\asymp \log \log p$, but we are not concerned with such lower-order terms since we have not even seriously optimized the exponent $1/4$ in Theorem \ref{thm:main}.

\section{Remarks and open problems}\label{sec:remarks}
We make a couple of remarks about our proof of Theorem \ref{thm:main}.
\begin{itemize}
    \item The union bound in Lemma \ref{lemm:fewcollisions} is one of the main bottlenecks for the value of the exponent $1/4$ in Theorem \ref{thm:main}.  Improving the argument around this lemma would    likely let one take $K$ to be a larger power of $R$, which in turn would let one increase $1/4$ (perhaps to $1/3$) in Theorem \ref{thm:main}.
    \item In proposition \ref{prop:refined-structure}, we can also obtain the extra property that each of $P,N$ is either empty or of size at least $100s$ (say), by splitting each dissociated set into $201$ parts and then absorbing up to $100$ elements of each of $P,N$ if $P,N$ are small.  This property was useful in an earlier version of our proof and may be of interest in the future.   
\end{itemize}

Our paper also leads to several open problems for future inquiry:
\begin{itemize}
    \item The most obvious open problem is improving the bound in Theorem \ref{thm:main}; a natural next goal would be a polynomial threshold (of the form $p^c$).  Even if our methods can be adapted to improve the exponent $1/4$ in Theorem \ref{thm:main}, it seems that neither our probabilistic toolbox nor our dissociated set machinery is suited for sets of polynomial size, so substantial new inputs would be necessary to reach a polynomial threshold.
    \item Our arguments for Theorem \ref{thm:main} show not only that there is some two-sided valid ordering of $A$ but that there are many such orderings.  It would be interesting to estimate the minimum possible number of two-sided valid orderings as a function of $|A|$ (and perhaps also $p$).
    \item The main result of \cite{kravitz} applies not only to the group $\mathbb{F}_p$ but also to all groups of the form $H_1 \times H_2$, where $H_1$ is an abelian group such that every subset of $H_1 \setminus \{0\}$ has a two-sided valid ordering and $H_2$ is an abelian group with no non-zero elements of order strictly smaller than $p$; one example is the group $\mathbb{Z}/2p\mathbb{Z} \cong \mathbb{Z}/2\mathbb{Z} \times \mathbb{Z}/p\mathbb{Z}$.  One could try to extend Theorem \ref{thm:main} to such groups.
    \item In a different direction, one might try to prove Graham's conjecture for very large sets, namely, for sets of size $|A| \geqslant p-f(p)$ for some function $f$ tending to infinity with $p$.  See \cite{HOS} and the references therein for more on Graham's conjecture for very large sets.
    \item Finally, we mention that nonabelian versions of Graham's conjecture, particularly in dihedral groups, have received some attention.  As in the abelian case, work prior to \cite{kravitz} concerned sets of size at most $12$.  Costa and Della Fiore \cite{CDF} then adapted the ideas of \cite{kravitz} to obtain results for sets of nearly logarithmic size in dihedral and dicyclic groups.  It seems more difficult to transfer the proof of Theorem \ref{thm:main} to nonabelian settings, and this could be a fruitful topic for future research.
\end{itemize}

\section*{Acknowledgements}
The first author gratefully acknowledges financial support from the EPSRC. The second author was supported in part by the NSF Graduate Research Fellowship Program under grant
DGE–203965.  We thank Ben Green for drawing our attention to the reference \cite{bourgain}.  We thank an anonymous referee for several helpful comments.

\end{document}